\documentclass[a4paper,11pt,noindent]{amsart}
\usepackage[centertags]{amsmath} 
\usepackage{amsfonts,amssymb,amsthm}
\usepackage{hyperref}
\usepackage{mathtools}

 \hypersetup{pdfmenubar=false, 
pdfnewwindow=true, 
colorlinks=false, 
linkcolor=blue, 
citecolor=blue, 
filecolor=magenta, 
urlcolor=cyan 
} 
\usepackage{graphicx}

\usepackage[english]{babel} \usepackage{newlfont} \usepackage{color}
\usepackage[body={15cm,21.5cm},centering]{geometry} \usepackage{fancyhdr}
 \usepackage{esint} \usepackage{enumerate}

 \newcommand{\Tr}{\mathrm{Tr}\,}

\newcommand{\red}[1]{#1}

\newcommand{\Rsym}{\R^{n\times n}_{\mathrm{sym}}}
\usepackage[active]{srcltx}

\newtheorem{theorem}{Theorem} \newtheorem{lemma}{Lemma}
\newtheorem{proposition}{Proposition} 
 
\newtheorem{definition}{Definition}

\theoremstyle{definition} 
 
\newtheorem{remark}{Remark}

\newcommand{\dist}{\operatorname{dist}} 

\newcommand{\e}{\varepsilon}

\newcommand\ecke{\mathop{\hbox{\vrule height 7pt width .3pt depth 0pt \vrule
height .3pt width 5pt depth 0pt}}\nolimits}

\newcommand{\R}{\mathbb{R}} 
\newcommand{\N}{\mathbb{N}} 
 
\renewcommand{\d}{\mathrm{d}} \renewcommand{\L}{\mathbb{L}}
\newcommand{\A}{\mathcal{A}} \newcommand{\M}{\mathcal{M}}
 \newcommand{\id}{\mathrm{Id}}

\newcommand{\G}{\mathcal G}

 \renewcommand{\H}{\mathcal{H}}
 
\renewcommand{\L}{{\mathcal L}}

\newcommand\wsto{\stackrel{*}{\rightharpoonup}}

\renewcommand{\div}{\mathrm{div}\,}

\newcommand{\rlim}[1]{\rho^{(#1)}}
\newcommand{\hh}[1]{h^{(#1)}}

\makeatletter
\newsavebox{\@brx}
\newcommand{\llangle}[1][]{\savebox{\@brx}{\(\m@th{#1\langle}\)}%
  \mathopen{\copy\@brx\kern-0.5\wd\@brx\usebox{\@brx}}}
\newcommand{\rrangle}[1][]{\savebox{\@brx}{\(\m@th{#1\rangle}\)}%
  \mathclose{\copy\@brx\kern-0.5\wd\@brx\usebox{\@brx}}}
\makeatother

\title{Michell truss type theories as a $\Gamma$-limit of optimal design in linear elasticity}
\date{\today}
\author[H. Olbermann] {Heiner Olbermann}
\address[Heiner Olbermann]{UCLouvain, Belgium}
\email{heiner.olbermann@uclouvain.be}

\date{\today} 

\begin{document}

\maketitle

\begin{abstract}
 We show how to derive (variants of) Michell truss theory in two and three dimensions rigorously as the  vanishing weight limit of optimal design problems in linear elasticity in the sense of $\Gamma$-convergence. We improve our results from \cite{olbermann2017michell} in that our treatment here includes the three dimensional case and that we allow for more general boundary conditions and applied forces. 
\end{abstract}

\section{Introduction}

In the present article we improve our results from 
\cite{olbermann2017michell}, where we have derived a certain form of Michell
truss theory as the vanishing weight limit of optimal design problems in linear
elasticity in a rigorous fashion.  The
improvement that we present here is twofold: First, we extend the analysis to
the three-dimensional case. Second, we allow for more general applied forces,
see Remark \ref{rem:mainthm} (v) below.

\medskip

We briefly
explain how the variational problem for finite values of the ``weight''
parameter that we will present in this introduction can be interpreted as
an optimal design problems in linear elasticity  in Section
\ref{sec:derivation-variational} of the appendix. For a \red{short} discussion of how our
limit problem can be considered as  the Michell truss problem (at least for the case of
two dimensions), see \red{Section \ref{sec:very-brief-pres} of the appendix and}  \cite{bouchitte2008michell}. Michell trusses, first devised
more than a century ago \cite{lviii185limits}, are a very popular model in
applied mathematics and engineering, see
e.g.~\cite{hemp1973optimum,rozvany2012structural,lewinski2018michell}.

\medskip

On a formal level, the relation between these variational models -- in both two and three dimensions -- had been observed by Allaire and Kohn  \cite{allaire1993optimal}.  As in \cite{olbermann2017michell}, our statements  should be viewed as  rigorous versions of their formal ones, in the framework of $\Gamma$-convergence.

\subsection{Notation}
\label{sec:notation}
Let $\L^d$ denote the $d$-dimensional Lebesgue measure, and  $\H^d$  the $d$-dimensional Hausdorff measure.
Let $E\subset \R^n$ be  either open or closed.
By $\mathcal M(E)$ (respectively $\mathcal M(E;\R^p)$)  we denote the space of Borel signed measures on $E$ (respectively $\R^p$-valued Borel measures). We denote the symmetric $d\times d$ matrices by $\R^{d\times d}_{\mathrm{sym}}=\{A\in\R^{d\times d}:A^T=A\}$.
The space $\mathcal M(E;\R^{d\times d}_{\mathrm{sym}})$ is the subspace of $\mu\in\mathcal M(E;\R^{d\times d})$ satisfying $\mu_{ij}=\mu_{ji}$ for all $i,j\in\{1,\dots,d\}$.

The set of non-negative Borel measures is denoted by $\mathcal M^+(E)$. 

For $\Omega\subset\R^n$ open and bounded with Lipschitz boundary, consider $\mu\in\mathcal M(\overline\Omega;\R^n)$ and $g\in \mathcal M(\overline\Omega)$. We say that $-\div\mu=g$ if
\[
  \int_{\overline\Omega}\sum \partial_{x_i}\varphi \d\mu_i=\int_{\overline\Omega}\varphi\d g\,
\]
for every compactly supported $\varphi\in C^1(\R^n)$. Put differently, the measures $\mu$ and $g$ are being viewed as measures on $\R^n$ with support on $\overline\Omega$.  When $\mu\in \mathcal M(\overline\Omega;\Rsym)$ and $g\in \mathcal M(\overline\Omega;\R^n)$, then $-\div \mu=g$ has to be understood row-wise.

\medskip

Let $1<p<\infty$, and $U\subset\R^n$ open. By $W^{-1,p}(U)$, we denote the dual of $W^{1,p'}_0(U)$, where $(p')^{-1}=1-p^{-1}$. It is well known that the following  norm on $W^{-1,p}(U)$ is equivalent to the norm as a dual space of $W^{1,p'}_0$, 
\[
  \|g\|_{W^{-1,p}(U)}=\inf\left\{\|\alpha\|_{L^p(U)}+\|\beta\|_{L^p(U)}:\,g=\alpha+\div\beta\right\}\,,
\]
where the equation $g=\alpha+\div\beta$ has to be understood in the sense of distributions,
\[
  \langle g,\varphi\rangle=\int_{U} \left(\varphi\alpha -\nabla\varphi\cdot \beta\right)\d x
  \]
  for all $\varphi\in C^1_c(U)$.

  By slight abuse of notation, we will write
  \[
    \mathcal M \cap W^{-1,p}(\overline\Omega)\equiv     \mathcal M(\overline\Omega)\cap W^{-1,p}(\R^n)\,.
    \]


\medskip

 For $\lambda>0$, we define $\tilde h_{\lambda}:\Rsym\to \R$ by
  \begin{equation}
  \tilde h_\lambda(\sigma)=\begin{cases}
    \frac{|\sigma|^2}{\sqrt{\lambda}}+\sqrt{\lambda} &\text{ if } \sigma\neq 0\\
    0 & \text{ if }\sigma=0\,,\end{cases}\label{eq:13}
  \end{equation}
where $|\cdot|$ denotes the Frobenius norm defined by $|A|^2=\Tr A^TA$.
  Now let $g\in \mathcal M\cap W^{-1,2}(\overline\Omega;\R^n)$,  to be thought of as the applied forces and normal component of the stress $\sigma$ at the boundary respectively (see Remark \ref{rem:mainthm} (iii) below).
  We define $\mathcal G_{\lambda,g}:\mathcal M(\overline\Omega;\Rsym)\to\R$ by
  \[
    \mathcal G_{\lambda,g}(\mu):=
    \begin{cases}
      \int_\Omega \tilde h_\lambda \left(\frac{\d\mu}{\d\L^n}\right)\d x & \text{ if
      }\mu\ll \L^n,\,\frac{\d\mu}{\d\L^n}\in L^2(\Omega;\Rsym) \text{ and } -\div \mu=g \text{ in }\overline\Omega\\
      +\infty & \text{ else.}
\end{cases}
\]
where $\mu\ll\L^n$ is the notation for $\mu$ being absolutely continuous with respect to $\L^n$, and $\frac{\d\mu}{\d\L^n}$ denotes the Radon-Nikod\'ym derivative of $\mu$ with respect to $\L^n$.

The variational functional $\mathcal G_{\lambda,g}$  defines an optimal design problem in linear elasticity, see Section \ref{sec:derivation-variational} of the appendix.

For $\sigma\in\Rsym$, let $\sigma_i$, $i=1,\dots,n$ denote the eigenvalues of $\sigma$, ordered such that $|\sigma_1|\leq|\sigma_2|\leq\dots\leq |\sigma_n|$.

\medskip

From now on, we will only be concerned with the case $n\in\{2,3\}$.
For $n=2$ we define
\[
    \rlim{2}(\sigma)=|\sigma_1|+|\sigma_2|
  \]
and for $n=3$, we let
\[
  \rlim{3}(\sigma)=\begin{cases} \frac{1}{2} \sqrt{
      (|\sigma_1|+|\sigma_2|)^2+\sigma_3^2} & \text{ if } |\sigma_1|+|\sigma_2|\leq
    |\sigma_3|\\
    \frac{1}{2\sqrt{2}} (|\sigma_1|+|\sigma_2|+|\sigma_3|) &\text{ else.}\end{cases} 
\]
Note that $\rlim{n}$ is positively one-homogeneous; hence for a $\Rsym$-valued Radon measure $\mu$, $\rlim{n}(\mu)$ can be defined as a Radon measure via

\[
  \rlim{n}(\mu)(A)=\int_A \rlim{n}\left(\frac{\d\mu}{\d|\mu|}\right)\d|\mu|\,,
\]
where $\frac{\d\mu}{\d|\mu|}$ is the Radon-Nikod\'ym derivative of $\mu$ with respect to its total variation measure $|\mu|$.
For $g\in \mathcal M(\overline\Omega;\R^n)$, we define $\mathcal G_{\infty,g}:\mathcal M(\overline\Omega;\Rsym)\to\R$ by 
\[
  \mathcal G_{\infty,g}(\mu)=
  \begin{cases}
    2\int_{\overline\Omega} \d \rlim{2}(\mu)&\text{ if } n=2, \,-\div \sigma=g\\
    \sqrt{2}\int_{\overline\Omega} \d \rlim{3}(\mu)&\text{ if } n=3, \,-\div \sigma=g\\
    +\infty &\text{ else.}
  \end{cases}
\]

\subsection{Statement of results}
\label{sec:statement-results}

We are ready to state our main theorem, namely the $\Gamma$-convergence
$\mathcal G_{\lambda,g_\lambda}\stackrel{\Gamma}{\to} \mathcal G_{\lambda,g}$
under the assumption of weak-* convergence of the applied forces,
$g_\lambda\wsto g$ with a $\lambda$-dependent control of
$\|g_\lambda\|_{W^{-1,2}}$. We recall that  $\Omega\subset\R^n$ is bounded open with Lipschitz boundary, where $n\in\{2,3\}$.

\begin{theorem}
  \label{thm:main}
   Assume that $g_\lambda\in \mathcal M\cap W^{-1,2}(\overline{\Omega};\R^n)$, $g\in \mathcal M(\overline\Omega;\R^n)$ such that
  $g_\lambda\wsto g$ in $\mathcal M(\overline\Omega;\R^n)$ and
  \begin{equation}\label{eq:9}
  \lambda^{-1/4}\|g_\lambda\|_{W^{-1,2}(\R^n;\R^n)}\to 0\,.
\end{equation}
  \begin{itemize}
  \item[(i)] (Compactness) Let $\mu_\lambda\in \M(\overline\Omega;\R^n)$ be a sequence such that 
    \[\limsup_{\lambda} \mathcal G_{\lambda,g_\lambda}(\mu_\lambda)<\infty\,.
    \]
    Then there exists $\mu\in \mathcal M(\overline\Omega;\Rsym)$ such that
      \begin{equation}\label{eq:4}
      \begin{split}
        \mu_\lambda\wsto \mu \text{ in }\mathcal M(\overline\Omega;\Rsym)\,.
      \end{split}
    \end{equation}
    \item[(ii)] (Lower bound) Let $\mu_\lambda\wsto \mu$ in $\M(\overline\Omega;\Rsym)$.
       Then we have that
      \[
        \liminf_{\lambda\to\infty}\mathcal
        G_{\lambda,g_\lambda}(\mu_\lambda)\geq \mathcal
        G_{\infty,g}(\mu)\,.
      \]
      \item[(iii)] (Upper bound) Let $\mu\in \mathcal M(\overline\Omega;\Rsym)$. Then there
    exists a sequence $\mu_\lambda$ in $L^2(\Omega;\Rsym)$ such that $\mu_\lambda\wsto \mu$ in $\M(\overline\Omega;\Rsym)$ and additionally
    \[
        \limsup_{\lambda\to\infty}\mathcal
        G_{\lambda,g_\lambda}(\mu_\lambda)\leq \mathcal
        G_{\infty,g}(\mu)\,.
      \]
  \end{itemize}
\end{theorem}

\begin{remark}
  \label{rem:mainthm}
  \begin{itemize}
\item[(i)] The proofs for the compactness and upper bound parts are fairly
        straightforward; the most interesting part is the lower bound part.
Here our proof is inspired by the theorem on lower semicontinuity for linear
growth functionals under PDE constraints by Arroyo-Rabasa, De Philippis and
Rindler \cite{arroyo2017lower}. Their work in turn builds on  the properties of
singular points of $\mathcal A$-free measures \cite{MR3549629}, the blow-up technique by Fonseca and M\"uller \cite{MR1218685},  and properties of the
        projection operator to $\mathcal A$-free functions proved by the same
        authors in
        \cite{MR1718306}.
      \item[(ii)]
It is the combination of the blow-up technique with the application of the
projection operator to $\mathcal A$-free measures that informs our choice of
assumptions for the convergence of the right hand sides, i.e.
\[g_\lambda\wsto
g\text{ in }\mathcal M(\overline\Omega;\R^n)\quad\text{ and
}\quad\lambda^{-1/4}g_\lambda\to 0\text{ in }W^{-1,2}(\R^n;\R^n)\,.
\]
These assumptions (or slightly weaker ones) are necessary in order for this
method of proof to work. Also in the proof of the upper bound the 
assumption on the growth of the $W^{-1,2}$ norms is heavily used.  It is not
clear to us if the statements remain true if this assumption is removed.  
\item[(iii)] The constraint equation $-\div\mu=g$ contains boundary conditions and applied forces at the same time.
To substantiate this claim, we consider the situation $\mu= a\L^n\ecke \Omega$ with $a\in W^{1,1}(\Omega;\R^n)$ and $g=b  
\L^n\ecke\Omega+c \H^{n-1}\ecke\partial\Omega$ with $b\in L^1(\Omega)$, $c\in L^1(\partial \Omega)$.   Then the equation $-\div \mu=g $ translates to
\[
  \left\{
    \begin{split}
      -\div a&= b \quad\text{ in }\Omega\\a\cdot n&=c\quad\text{ on }\partial\Omega\,,
    \end{split}
  \right.
\]
where $n$ denotes the unit outer normal to $\partial\Omega$.
\item[(iv)] As an example for the approximation of applied forces,
      consider a point force $g=\sum_i
      g_i\delta_{x_i}$ (with $x_i\in\Omega$, $g_i\in\R^n$, and where $\delta_x$ denotes the Dirac measure supported  in $\{x\}$) which is permitted in the
      Michell truss problem in the sense that there exists a measure
      $\mu\in\mathcal M(\overline\Omega;\Rsym)$
      satisfying $g=-\div \mu$. Meanwhile,  any such $\mu$ cannot be absolutely continuous with respect to $\L^n$ with $\d\mu/\d\L^n\in L^2$, which implies that such $\mu$ is not permissible in the linear
      elasticity problems in the sense that
      $\G_{\lambda,g}(\mu)=+\infty$. A suitable
      approximation of the limit problem is given by some sequence $g_j$ satisfying $g_j\in W^{-1,2}$ and \eqref{eq:9}; this can be easily achieved, e.g., by mollification.
\item[(v)] In \cite{olbermann2017michell}, we only allowed for right hand sides
  of a very particular form. Namely, the stresses $\mu$ for the optimal design
  problems had to be solutions of boundary value problems
    \begin{equation}\label{eq:24}
      \left\{     \begin{split}
-\div \mu&=0\quad\text{ in }\Omega\\ \mu\cdot n&= \tilde g_\lambda\quad\text{ on
}\partial\Omega
\end{split}\right.
\end{equation}
where $\tilde g_\lambda\in W^{-1/2,2}(\partial\Omega)$, and $n$ denotes the unit
outer normal to $\partial \Omega$. Additionally, we required that $\Omega$ be simply connected and piecewise $C^2$. By (iii) above, this is
just a special case of the right hand sides that we are treating here. It was our method of proof that limited us to right hand sides
that correspond to \eqref{eq:24} in
\cite{olbermann2017michell}. There we reformulated  the problem as one for $BV$
functions, which moreover is only possible in two dimensions. 
\red{
\item[(vi)] One can  interpret the penalization parameter $\lambda$ as a
  Lagrange multiplier enforcing a constraint on the mass of the elastic body
  in the minimization problem for the compliance defined by
  $\G_{\lambda,g_\lambda}$ (see Section
  \ref{sec:derivation-variational} of the appendix). The connection between the
  constrained problem and the one we are considering is however only formal, see
  the discussion in \cite{MR820342}. As can be seen straightforwardly seen from  an
  inspection of our proof of the upper bound,  recovery sequences for the
  stresses  $\mu_\lambda$
  will
  typically be non-zero on a set of measure
  $O(\lambda^{-1/2})$  for singular limits $\mu$ that are singular with respect
  to the Lebesgue measure. In terms of the optimal design problem, the set where the stress is non-zero has to be understood
  as  occupied by  the elastic material, while the set where the stress is 0
  should be thought of as ``holes''. 
\item[(vii)] In \cite{MR2022556}, a general formula is proposed for describing the $\Gamma$-limit of optimal
design problems for vanishing volume fractions, also for non-linear cases. This
formula agrees with ours for the case we treat here.

}
          \end{itemize}
\end{remark}

\subsubsection*{Notation} The symbol $C$ will be used as follows: A statement $f\leq C(a,b,\dots)g$ has to be read as ``there exists a constant $C>0$ only depending on $a,b,\dots$ such that $f\leq C g$''. The value of $C$ may change from one inequality to the next. When it is clear on which quantities the constant depends, we also write $f\lesssim g$ in this situation.

\section{Preliminaries}

\subsection{$\mathcal A$-free singular measures}
\label{sec:mathcal-a-free}
Let $\mathcal A$ denote a linear partial differential operator of order $k\in\N$,
\[
  \mathcal A=\sum_{|\alpha|\leq k}A_\alpha\partial^\alpha\,,
\]
where $A_\alpha\in \R^{p\times m}$ for every multiindex $\alpha\in \N^n$ with
$\partial^\alpha=\partial_{x_1}^{\alpha_1}\dots \partial_{x_n}^{\alpha_n}$ and
$|\alpha|=\sum_{i=1}^n \alpha_i$. We define the principal
symbol of $\mathcal A$, $\mathbf A^k:\R^n\to \R^{p\times m}$,  by setting
\[
  \mathbf{A}^k(\xi)=\sum_{|\alpha|=k} A_\alpha \xi^\alpha\,,
\]
where $\xi^\alpha=\xi_1^{\alpha_1}\dots \xi_n^{\alpha_n}$.
In the following definition, $S^{n-1}=\{x\in\R^n:|x|=1\}$.
\begin{definition}
  The wave cone associated to a differential operator $\mathcal A$ as above is
  defined by
  \[
    \Lambda_{\mathcal A} =\bigcup_{\xi\in S^{n-1}} \mathrm{Ker} \mathbf A^k(\xi)\,.
    \]
  
  \end{definition}

We will only be interested in the case $\mathcal A=\div$, acting on  measures with values in $\Rsym$ (i.e., $p=n$, $m=n(n+1)/2$). In this case we obtain
\[
  \Lambda_\div=\{A\in \Rsym: \mathrm{Rk}\, A\leq n-1\}\,\,
\]
where $\mathrm{Rk}\, A$ denotes the rank of $A$.

\begin{definition}
    The operator $\mathcal A$ is said to satisfy the constant-rank condition if
    there exists $r\in \N$ such that
    \[
      \mathrm{Rk}\, \mathbf A^k(\xi)=r\quad\text{ for all }\xi\in S^{n-1}\,.
      \]
  \end{definition}
One easily verifies that the constant-rank condition is fulfilled for $\mathcal A=\mathrm{div}$ with $r=1$.

\medskip

The structure of $\A$-free singular measures by De Philippis and Rindler yields
in particular the following result:
\begin{theorem}[See \cite{MR3549629}]
\label{thm:DPR}  Let $\Omega\subset \R^n$ be open and $\mu\in \M(\Omega;\R^m)$ satisfy $\A
  \mu=\sigma$, where $\sigma\in \M(\Omega;\R^p)$. Then for $|\mu^s|$ a.e. $x_0$, we have that
  \[
    \frac{\d\mu}{\d|\mu|}(x_0)\in\Lambda_{\A}\,.
    \]
\end{theorem}

\subsection{Generalized Young measures}
Generalized Young measures -- roughly speaking -- are dual objects to functions with
linear growth at infinity. They have been introduced by DiPerna and Majda
\cite{diperna1987oscillations}. Here we follow closely the approach by
Kristensen and Rindler \cite{kristensen2010characterization}, \red{which in turn is
based on the work by Alibert and Bouchitt\'e \cite{alibert1997nonuniform}}. In
comparison to \cite{kristensen2010characterization}, we drop the
dependence of test functions on a variable $x\in\Omega$, since we will not need
this for our purpose.

\medskip

First we define a suitable set of functions with linear growth at infinity.
For $f\in C(\R^m)$ and $\xi\in B(0,1)\subset\R^m$, let
\[
  Tf(\xi)=(1-|\xi|)f\left(\frac{\xi}{1-|\xi|}\right)\,.
\]
We define
\[
  \mathbf E(\R^m)=\left\{f\in C(\R^m): Tf \text{ extends to a continuous function
      on }\overline {B(0,1)}\subset\R^m\right\}\,.
  \]

\begin{definition}
  A generalized Young measure $\nu$ parametrized by a set $\Omega\subset\R^n$ with values in $\R^m$ is a triple
  $(\nu_x,\lambda_\nu,\nu_x^\infty)$, where
  \begin{itemize}
  \item $(\nu_x)_{x\in\Omega}$ is a family of probability measures on $\R^m$,
  \item $\lambda_\nu\in\mathcal M^+(\overline\Omega)$ is a non-negative measure
    \item $(\nu_x^\infty)_{x\in\Omega}$ is a family of probability measures on $S^{m-1}$
    \end{itemize}
    such that $x\mapsto \nu_x$ is weakly * measurable with respect to $\L^n$,
    $x\mapsto \nu_x^\infty$ is weakly * measurable with respect to
    $\lambda_\nu$, and $\left(x\mapsto \langle |\cdot|,\nu_x\rangle\right)\in L^1(\Omega)$.
\end{definition}

In the above definition, weak * measurability means that for every
$f\in\mathbf E(\R^m)$, we have that
$x\mapsto \langle f(\cdot),\nu_x\rangle$ is $\L^n$-measurable, and
$x\mapsto \langle f(\cdot),\nu_x^\infty\rangle$ is $\lambda_\nu$-measurable.
The duality between generalized Young measures and functions $f\in \mathbf{E}(\R^m)$ is
defined by
\[
  \llangle f,\nu\rrangle=\int_\Omega \langle f,\nu_x\rangle \d x
  +\int_{\overline{\Omega}} \langle f^\infty,\nu_x^\infty\rangle\d\lambda_\nu(x)\,,
\]
where $f^\infty$ denotes the recession function of $f$,
\[
  f^\infty(\xi)=\limsup_{t\to\infty}\frac{f(t\xi)}{t}\,.
  \]

\medskip

By Jensen's inequality, we
have for convex $f$  that
\begin{equation}
\label{eq:10}  \begin{split}
    f\left( \langle \id, \nu_x\rangle \right)&\leq \langle f,\nu_x\rangle \quad
    \text{ for $\L^n$ a.e. } x\in\Omega\\
    f^\infty\left( \langle \id, \nu_x^\infty\rangle \right)&\leq \langle
    f^\infty,\nu_x^\infty\rangle \quad \text{ for $\lambda_\nu$ a.e. } x\in\Omega\,.
  \end{split}
\end{equation}

\medskip

By the Radon-Nikod\'ym Theorem, we may decompose any measure $\mu\in \mathcal
M(\overline\Omega;\R^m)$ into two parts, the one regular with respect to $\L^n$, and its
singular part:
\[
  \mu=\frac{\d\mu}{\d\L^n}\L^n+\mu^s\,.
  \]
Such a measure $\mu\in \mathcal M(\overline\Omega;\R^m)$ can be identified with a Young
measure $\delta[\mu]$ via
\[
  (\delta[\mu])_x=\delta_{\frac{\d \mu}{\d\L^n}(x)}, \quad
\lambda_{\delta[\mu]}=|\mu^s|,\quad (\delta[\mu])_x^\infty=\delta_{\frac{\d
    \mu^s}{\d|\mu^s|}(x)}
\]

We say that a sequence of Young measures $(\nu_j)_{j\in\N}$ converges weakly *
to a Young measure $\nu$ if for every (globally) Lipschitz function $f:\R^m\to\R$, we have that
$\llangle f,\nu_j\rrangle\to \llangle f,\nu\rrangle$.
In this case we write $\nu_j\stackrel{\mathbf Y}{\to}\nu$.

\medskip

We say that a sequence of measures $\mu_j$ generates a Young measure $\nu$ if we
have $\delta[\mu_j]\stackrel{\mathbf Y}{\to} \nu$ weakly * as Young measures.

\medskip

Finally, we have the following compactness result for generalized Young measures:
\begin{lemma}[\cite{kristensen2010characterization}, Corollary 2]
\label{lem:Youngcomp}  Let $(\nu_j)_{j\in\N}$ be a sequence of generalized Young measures such that
  the functions
  $x\mapsto \langle|\cdot|,(\nu_j)_x\rangle$ are uniformly bounded in
  $L^1$ and $\lambda_{\nu_j}(\overline\Omega)$ is uniformly bounded. Then there
  exists a generalized Young measure $\nu$ such that $\nu_j\stackrel{\mathbf
    Y}{\rightarrow} \nu$.
\end{lemma}

\subsection{$\mathcal A$-quasiconvexity}
Let $\mathcal A=\sum_{|\alpha|\leq k}A_\alpha \partial^\alpha$ be a partial
differential operator as in Section \ref{sec:mathcal-a-free} above. Let $Q=(-1/2,1/2)^n$ be the unit
cube in $\R^n$. The smooth $Q$-periodic functions with values in $\R^m$ are
denoted by $C^\infty_{\mathrm{per}}(Q;\R^m)$.
\begin{definition}
  \begin{itemize}
  \item[(i)] A Borel function $f:\R^m\to \R$ is said to be $\mathcal A$-quasiconvex
    if for every $\xi \in \R^m$ and every 
    $\varphi\in C^\infty_{\mathrm{per}}(Q;\R^m)$ satisfying
    \[
      \mathcal A\varphi=0 \qquad \text{ and }\qquad \int_Q\varphi(x)\d x=0
    \]
    we have that
    \[
      \int_Q f(\xi+\varphi(x))\d x\geq f(\xi)\,.
    \]
  \item[(ii)]
    For a Borel function $f:\R^m\to\R$, the $\mathcal A$-quasiconvexification of
    $f$, $Q_{\mathcal A}f$, is given by
          \[
      Q_{\mathcal A}f(\xi)=\inf\left\{\int_Q f(\xi+\varphi(x))\d x:\,\varphi\in
      C^\infty_{\mathrm{per}}(Q;\R^m) \cap \mathrm{Ker}\mathcal A\text{  with }\int_Q \varphi(x)\d x=0\right\}\,.
    \]
    
  \end{itemize}

\end{definition}

Since we will only be interested in the case $\mathcal A=\div$, we will write  $Qf\equiv Q_\div f$.

\medskip

In the following lemma, functions in $L^p(Q;\R^m)$ are identified with their $Q$-periodic extensions. Furthermore, let $W^{1,p'}_{\mathrm{per}}(Q)$ denote the $Q$-periodic functions in  $W^{1,p'}_{\mathrm{loc}}(\R^n)$, and let $W^{-1,p}_{\mathrm{per}}(Q)$ denote its dual (where $(p')^{-1}=1-p^{-1}$). For later usage, we remark that
\[
  \|f\|_{W^{-1,p}_{\mathrm{per}}(Q)}\lesssim \|f \chi_Q\|_{W^{-1,p}(Q)}\,.
  \]

  \begin{lemma}[\cite{MR1718306}, Lemma 2.14]
    \label{lem:Aproj}
    Let $\mathcal A$ be a first order differential operator as above that satisfies the constant rank condition and 
    $1<p<\infty$. There exists an operator $\mathcal P_{\mathcal A}:L^p(Q;\R^m)\to
    L^p(Q;\R^m)$ and a constant $C=C(p)>0$ such that
    \[
      \mathcal A \mathcal P_{\mathcal A} \varphi=0,\quad \int_Q \mathcal P_{\mathcal A} \varphi\,\d x=0, \quad
      \|\varphi-\mathcal P_{\mathcal A} \varphi\|_{L^p(Q;\R^m)}\leq C \|\mathcal A
      \varphi\|_{W^{-1,p}_{\mathrm{per}}(Q)}
    \]
    for every $\varphi\in L^p(Q;\R^m)$ with $\int_Q \varphi(x)\d x=0$.
  \end{lemma}

  \subsection{Tangent measures}
  \label{sec:tangent-measures}
  The notion of tangent measures is due to  Preiss \cite{preiss1987geometry}.
  We will only need one fact about tangent measures, for which it will not even
  be necessary to mention the definition. For $x_0\in\R^n$, $r>0$, let
  $T^{(x_0,r)}(x)=r^{-1}(x-x_0)$. The push-forward of a measure $\mu\in \mathcal
  M(\R^n)$ by $T^{(x_0,r)}$ is given by 
  \[
    T^{(x_0,r)}_\# \mu(A)=\mu (x_0+rA)\,.
    \]
The fact that we are going to use is that for $\L^n$ almost every $x_0\in\R^n$,
there exists a sequence $r_j\downarrow 0$ such that
\[
  T^{(x_0,r_j)}_\# \mu\wsto \frac{\d\mu}{\d\L^n}(x_0)\L^n\,.
  \]
  This follows e.g.~from Theorem 2.44 in \cite{MR1857292} in combination with
  the Radon-Nikod\'ym differentiation theorem.
\subsection{Quasiconvexification of $\tilde h_\lambda$}
One of the main ingredients for the derivation of our convergence result are the
known relaxations of the functionals $\G_{\lambda,g}$ for $\lambda<\infty$. A
proof of the 
following statement can be found in  \cite{allaire1993optimal} (see also \cite{MR820342,MR1859696,allaire1997shape}).

\begin{theorem}
  The $\div$-quasiconvexification of $\tilde h_\lambda$, $h_\lambda=Q\tilde h_\lambda$, is given by the following
  formulas:
  \begin{itemize}
  \item If $n=2$ then
    \[
       h_\lambda(\tau)= \begin{cases} \lambda^{-1/2}|\tau|^2+\lambda^{1/2} &\text{ if
        }\rho^{(2)}(\tau)\geq
        \sqrt{\lambda}\\
        2 
        \left(\rho^{(2)}(\tau)-\lambda^{-1/2}|\det\tau|\right) &\text{
          else.}\end{cases}
    \]
    \item If $n=3$, then
    \[
       h_\lambda(\tau)= \begin{cases} \lambda^{-1/2}|\tau|^2+\lambda^{1/2} &\text{ if
        }\rho^{(3)}(\tau)\geq
        \sqrt{\lambda}\\
        2  \left(\sqrt{ (|\tau_1|+|\tau_2|)^2+\tau_3^2}-\lambda^{-1/2}|\tau_1\tau_2|\right) & \text{ if }
        \rho^{(3)}(\tau)\leq\sqrt{\lambda} \text{ and }|\tau_1|+|\tau_2|\leq
        |\tau_3|\\
        h_\lambda^*(\tau)&\text{
          else,}\end{cases}
    \]
    where
    \[
      h_\lambda^*(\tau)=\sqrt{2} (|\tau_1|+|\tau_2|+|\tau_3|)+\lambda^{-1/2}\left(\frac12
      |\tau|^2-(|\tau_1\tau_2|+|\tau_1\tau_3|+|\tau_2\tau_3|)\right)\,.
      \]
  \end{itemize}
\end{theorem}

Obviously we have the following pointwise convergences:
If $n=2$, then
\[
  \lim_{\lambda\to\infty}
   h_\lambda(\tau)=2\rho^{(2)}(\tau)=:h^{(2)}(\tau)
  \]
  and if $n=3$, then  
  \[
     \lim_{\lambda\to\infty}
     h_\lambda(\tau)=\sqrt{2}\rho^{(3)}(\tau)=:h^{(3)}(\tau)\,.
  \]
Whenever we make statements that are true for $n\in\{2,3\}$, we also write
$h\equiv \hh{n}$.

We consider the divergence operator on symmetric matrices (which may
be identified with $\R^{n(n+1)/2}$). We have already noted that the wavecone  is given by 
\[
  \Lambda_{\div}=\{A\in \R^{n\times n}: A=A^T, \mathrm{Rk}\, A\leq n-1\}\,.
\]

This readily implies that the restriction of $\hh{n}$ to $\Lambda_{\div}$
(which is obtained by setting $\tau_1=0$) is
given by, for $n=2$, 
\[
  \hh{2}|_{\Lambda_\div}(\tau)=2|\tau_2|\,,
  \]
and for $n=3$ by
  \[
    \hh{3}|_{\Lambda_\div}(\tau)= 2  \sqrt{
      \tau_2^2+\tau_3^2} \,.
  \]

Of course, the right hand side of the last two equations is defined on all of
$\Rsym$. We denote it by $H^{(n)}$,
\[
  H^{(2)}(\tau)=2|\tau_2|,\qquad H^{(3)}(\tau)=2  \sqrt{
    \tau_2^2+\tau_3^2} \,.
\]
Again we
write $H\equiv H^{(n)}$ whenever statements hold simultaneously for $n=2$ and $n=3$.

\begin{lemma}
  \label{lem:Hestim}The function $H$ is convex, and for every $\lambda>0$, we have that
\[
  H(\tau)\leq h_\lambda(\tau)\,.
\]
      
    \end{lemma}
    \begin{proof}
The convexity is straightforward from the formulas above. Concerning the
inequality, for $n=2$, this is obvious. For $n=3$, we insert $\tau_1=0$, verify the inequality, and then verify by a direct
      computation that
      $\partial_{\tau_1} h_\lambda(\tau)\geq 0$ (for a.e.~$\tau$).
    \end{proof}

\section{Proof of the lower bound}



By the Radon-Nikod\'ym theorem, we have the decomposition of the limit measure
$\mu\in\mathcal M(\Omega;\Rsym)$ into one part that is singular with respect to the Lebesgue
measure, and the regular part,
\[
  \mu=\frac{\d\mu}{\d\L^n}\L^n +\mu^s\,.
  \]
Using the blow-up technique, we will prove the lower bound at regular and at
singular points separately.

\subsection*{Lower bound at singular points}
\begin{proposition}
  \label{prop:lowersing}
      Let $u_j\in L^2(\Omega;\Rsym)$, $\mu\in \mathcal M(\overline\Omega;\Rsym)$ with   $u_j\L^n\ecke\Omega\wsto \mu$ and  
      $\div \mu\in \M(\overline\Omega;\R^n)$. For $|\mu^s|$ almost every $x_0$,
      with $\frac{\d\mu}{\d|\mu|}(x_0)=\xi$, we have that
      \[
        \lim_{r\to 0}\liminf_{\lambda\to\infty} \frac{1}{|\mu|(Q(x_0,r))}\int_{Q(x_0,r)} h_\lambda(u_\lambda)\d x\geq h(\xi)\,.
      \]
    \end{proposition}

    \begin{proof}
By Lemma \ref{lem:Youngcomp}      we have -- possibly after passing to a subsequence -- that $u_\lambda\L^n$ generates a
Young measure $\nu$, $u_\lambda\L^n\stackrel{\mathbf Y}{\to} \nu$. 
      Now by Lemma \ref{lem:Hestim}, we have for $\lambda_\nu$ almost every $x_0$, 

              \begin{equation}\label{eq:18}
        \begin{split}
        \lim_{r\to 0}\liminf_{\lambda\to\infty} &\frac{1}{|\mu|(Q(x_0,r))}\int_{Q(x_0,r)} h_\lambda(u_\lambda)\d x\\
        &\geq \lim_{r\to 0}\liminf_{\lambda\to\infty} \frac{1}{\lambda_\nu(Q(x_0,r))}\int_{Q(x_0,r)}  H(u_\lambda)\d x\\
        &= \lim_{r\to 0} \frac{1}{\lambda_\nu(Q(x_0,r))}\left(\int_{Q(x_0,r)} \langle H,\nu_x\rangle\d x+\int_{Q(x_0,r)} \langle H,\nu_x^\infty\rangle\d\lambda_\nu\right)\\
        &=\langle H, \nu_{x_0}^\infty\rangle 
      \end{split}
    \end{equation}

Here the limit $\lim_{r\to 0}$ has to be understood as the choice of a sequence
$(r_i)_{i\in\N}$, $r_i\downarrow 0$, such that $\lambda_\nu(\partial
Q(x_0,r_i))=0$ for all $i\in\N$, in order to justify the penultimate equality in
\eqref{eq:18}.
    In the last equality of \eqref{eq:18}, we have used that
    \[
      \begin{split}
      \frac{\d(\langle H,\nu_x^\infty\rangle \lambda_\nu)}{\d \lambda_\nu}(x_0)&=\langle H, \nu_{x_0}^\infty\rangle \text{ for }\lambda_\nu \text{ a.e. }x_0\\
      \frac{\d(\langle H,\nu_x\rangle \L^n)}{\d \lambda_\nu}(x_0)&=0 \text{ for }\lambda_\nu \text{ a.e. }x_0
    \end{split}
    \]
Using equation \eqref{eq:10} and the convexity of $H$,  we obtain
\[
  \begin{split}
  \lim_{r\to 0}\liminf_{\lambda\to\infty}
  \frac{1}{|\mu|(Q(x_0,r))}\int_{Q(x_0,r)} h_\lambda(u_\lambda)\d x&\geq \langle
  H,\nu_{x_0}^\infty\rangle \\
  &\geq H(\langle\id,\nu_{x_0}^\infty\rangle)\\
  &= H(\xi)\\
  &=h(\xi)\,.
\end{split}
  \]
  In the last equality, we have used the fact that $\xi\in \Lambda_{\div}$ for
  $\lambda_\nu$ almost every $x_0$ by
  Theorem \ref{thm:DPR}.
  \end{proof}
    

    \bigskip

    \subsection*{Lower bound for regular points}
    
    Our proof can be viewed as an adaptation of the proof of   Lemma 2.15 in \cite{arroyo2017lower}\red{, which itself is a variation of the proof of Proposition 3.1 in \cite{fonseca2004quasiconvexity}}. 
Even
though we will only need the case $\mathcal A=\div$, we will prove the lower
bound at regular points in a slightly more general setting. Namely, let
$\mathcal A$ be a first order linear partial differential operator.
Let  $\mathcal P\equiv \mathcal P_{\mathcal A}$ denote the projection operator onto
the (mean-free) $\mathcal A$-free functions from Lemma \ref{lem:Aproj}.

In the following, let $p>1$, $q>0$.  Furthermore, let  $f_\lambda:\R^m\to\R$ be $\mathcal A$-quasiconvex and locally Lipschitz with the estimate
\begin{equation}
  \label{eq:2}
        |\nabla f_\lambda(A)|\lesssim 1+\frac{|A|^{p-1}}{\lambda^q}\quad \text{ for a.e. }A\,.
      \end{equation}
This assumption translates into the estimate
\begin{equation}\label{eq:5}
        |f_\lambda(A)-f_\lambda(B)|\lesssim |A-B|\left(1+\frac{|A|^{p-1}+|B|^{p-1}}{\lambda^q}\right)\,.
          \end{equation}
          We write
          \[
            f(\xi)=\liminf_{\lambda\to\infty} f_\lambda(\xi)\,.
          \]

          One further property that we are going to assume (and that is valid in
          the case $f_\lambda=h_\lambda$ that we will be interested in later) is
            \begin{equation}\label{eq:8}
            f_\lambda(sA)\leq C s f_\lambda(A)
          \end{equation}
          for $s\leq 1$.
          
          \begin{proposition}
\label{prop:lowerreg}            Let $\mathcal A$ be a first order linear differential operator satisfying the constant rank condition,
            $\mu_\lambda$ a sequence in  $\mathcal M(\overline Q;\R^m)$ with
            $|\mu_\lambda|\ll \L^n$, $\frac{\d\mu_{\lambda}}{\d\L^n}\in
            L^p(Q;\R^m)$ and $\xi\in\R^m$, $\mu:=\xi\L^n$ such that
            \[
              \begin{split}
                \mu_\lambda-\mu&\wsto 0\quad \text{ in }\mathcal M(\overline Q;\R^m)\\
                \A(\mu_\lambda-\mu)&\wsto 0\quad \text{ in }\mathcal M(\overline
                Q;\R^m)\\
                \lambda^{-q/p}\A(\mu_\lambda-\mu)&\to 0 \quad \text{ in }
                W^{-1,p}( Q;\R^m)\\
                \lambda^{-q/p}\frac{\d\mu_\lambda}{\d\L^n}& \quad\text{ is bounded in } L^p(Q;\R^m)                \,.
              \end{split}
              \]
              Then
              \[
                f(\xi)\leq \liminf_{\lambda\to\infty}\int_Q f_\lambda\left(\frac{\d\mu_\lambda}{\d\L^n}\right)\d x\,.
                \]
\end{proposition}
\begin{proof}
After taking subsequences we may assume that the right hand side is  a limit,
and that it is finite. Let $\eta\in C^\infty_c(\R^n)$ with $\int_{\R^n}\eta\,\d x=1$
and $\eta_\e\:=\e^{-n}\eta(\cdot/\e)$.
For $k\in\N$, choose $\e(\lambda,k)\downarrow 0$ as $\lambda\to \infty$. Now set
for $x\in\R^n$,
\[
  \begin{split}
  w_{\lambda,k}(x)&:=\int_{\R^n}\eta_{\e(\lambda,k)}(x-\cdot)\d (\mu_\lambda-\mu)\\
  &=\left(\eta_{\e(\lambda,k)}*(\mu_\lambda-\mu)\right)(x)\,.
\end{split}
  \]
We have that
  \begin{equation}
  \begin{split}
    |w_{\lambda,k}\L^n-(\mu_\lambda-\mu)|(\overline Q)&\downarrow 0 \quad \text{ as }
    \lambda\to 0\,.
  \end{split}\label{eq:14}
  \end{equation}
Let $\delta\in(0,1)$, and set
\[
  Q_k:=\{x\in
  Q:\dist(x,\partial Q)>\frac\delta k\}\
  \]
and let $\varphi_k$ be associated test functions with $\varphi_k=1$ on $Q_k$ and
$\varphi_k=0$ on $Q\setminus Q_{k+1}$. We set
\[
\hat w_{\lambda,k}=\varphi_kw_{\lambda,k}\,.
\]
Now let
\begin{equation*}
  \begin{split}
    \bar w_{\lambda,k}&=\hat w_{\lambda,k}-\fint_Q \hat w_{\lambda,k}\d x\\
  \tilde w_{\lambda,k}&=\mathcal P_{\mathcal A} \bar w_{\lambda,k}\,.
\end{split}
\end{equation*}
Note that $\bar w_{\lambda,k}\in C^\infty_{\mathrm{per}}(Q;\R^m)$.
For every $\lambda,k$ we have by the $\mathcal A$-quasiconvexity of $f_\lambda$ that

    \begin{equation}\label{eq:13}
  f_\lambda(\xi)\leq \int_Q f_\lambda(\xi+\tilde w_{\lambda,k})\d x\,.
\end{equation}

Using \eqref{eq:5} and H\"older's inequality, we have
\begin{equation}
  \begin{split}
  \int_Q f_\lambda(\xi+\tilde w_{\lambda,k})\d x
  \leq &\int_Q f_\lambda(\xi+\bar w_{\lambda,k})\d x +C \|\mathcal P\bar w_{\lambda,k}-\bar w_{\lambda,k}\|_{L^1(Q)}\\
  &\quad +\frac{C}{\lambda^q}\left(\|\bar w_{\lambda,k}\|^{p-1}_{L^p(Q)}+\|\mathcal P\bar w_{\lambda,k}\|^{p-1}_{L^p(Q)}\right)\|\mathcal P\bar w_{\lambda,k}-\bar w_{\lambda,k}\|_{L^p(Q)}\,,
\end{split}\label{eq:12}
\end{equation}

We claim that for the error terms on the right hand side vanish in the limit $\lambda\to 0$. Indeed, we have for any $\bar p\in (1,\frac{n}{n-1})$, 
  \begin{equation}\label{eq:6}
  \begin{split}
    \limsup_{\lambda\to \infty} \|\mathcal P\bar w_{\lambda,k}-\bar w_{\lambda,k}\|_{L^1(Q)}&\stackrel{\text{H\"older}}{\lesssim} \limsup_{\lambda\to \infty} \|\mathcal P\bar w_{\lambda,k}-\bar w_{\lambda,k}\|_{L^{\bar p}(Q)}\\
   &\lesssim \limsup_{\lambda\to \infty}\|\mathcal A \bar w_{\lambda,k}\|_{W^{-1,\bar p}(Q)}\\
    \limsup_{\lambda\to\infty}\lambda^{-q/p}\|\mathcal P\bar w_{\lambda,k}-\bar w_{\lambda,k}\|_{L^p(Q)}&\lesssim
   \limsup_{\lambda\to\infty} \lambda^{-q/p}\|\mathcal A \bar w_{\lambda,k}\|_{W^{-1, p}(Q)}\,.
  \end{split}
\end{equation}
Furthermore
\[
  \begin{split}
  \mathcal A \bar w_{\lambda,k}&=\mathcal A  (\varphi_k \eta_\e* (\mu_\lambda-\mu))\\
  &=\varphi_k\eta_\e*\mathcal A   (\mu_\lambda-\mu)+\eta_\e*(\mu_\lambda-\mu)\mathcal A\varphi_k\,.
\end{split}
  \]
From this expansion and  $\mu_\lambda-\mu\wsto 0$, $\A(\mu_\lambda-\mu)\wsto 0$
we obtain  $\A \bar w_{\lambda,k}\wsto 0$ in $\mathcal M(Q;\R^m)$. By the
compact embedding $\mathcal M(Q;\R^m)\subset W^{-1,\bar p}(Q;\R^{m})$, we
have that $\|\mathcal A \bar w_{\lambda,k}\|_{W^{-1,\bar p}}\to 0$. 
From this and the assumption on the vanishing of $\lambda^{-p/q}\mathcal
A(\mu_\lambda-\mu)$ in $W^{-1,p}$ we obtain 
   our claim that the right hand sides in \eqref{eq:6} vanish too.

 Next we have, again using \eqref{eq:5},
   \begin{equation}\label{eq:16}
   \begin{split}
   \int_Q f_\lambda(\xi+\bar w_{\lambda,k})\d x&\leq
   \int_Q f_\lambda(\xi+ \hat w_{\lambda,k})\d x\\
   &\quad + C\left|\fint \hat
       w_{\lambda,k}\d x\right|+C\lambda^{-q}\left(\|\hat w_{\lambda,k}\|_{L^p}^{p-1}+\|\bar w_{\lambda,k}\|_{L^p}^{p-1}\right)\left|\fint \hat w_{\lambda,k}\d x\right|\,,
   \end{split}
 \end{equation}
By the boundedness assumption on $\lambda^{-q/p}\d(\mu_\lambda-\mu)/\d\L^n$ in $L^p$,  the error terms on the right hand side converge to 0 in the limit $\lambda\to 0$. Next,
   \begin{equation}\label{eq:17}
     \begin{split}
   \int_Q f_\lambda(\xi+ \hat w_{\lambda,k})\d x &=
   \int_{Q_k} f_\lambda\left(\frac{\d\mu_\lambda}{\d\L^n}\right)\d x +\int_{Q_{k+1}\setminus
     Q_k} f_\lambda(\xi+\varphi_k \eta_\e*(\mu_\lambda-\mu))\d
   x\\
   &\quad+\int_{Q\setminus Q_{k+1}} f_\lambda(\xi)\d x\,.
 \end{split}
 \end{equation}

 Combining \eqref{eq:13}, \eqref{eq:12}, \eqref{eq:16} and \eqref{eq:17},  and taking the limit $\lambda\to\infty$, we obtain
 \[
   f(\xi)\leq \liminf_{\lambda\to\infty} \left(\int_{Q_k} f_\lambda\left(\frac{\d\mu_\lambda}{\d\L^n}\right)+
     \int_{Q_{k+1}\setminus Q_{k}}
     f_\lambda(\xi+\varphi_kw_{\lambda,k})\d x+\int_{Q\setminus Q_{k+1}} f_\lambda(\xi)\d x\right)\,.
 \]
 Reordering and using $f_\lambda\geq 0$ yields
   \begin{equation}\label{eq:7}
 |Q_{k+1}| f(\xi)\leq \liminf_{\lambda\to\infty} \left(\int_{Q} f_\lambda\left(\frac{\d\mu_\lambda}{\d\L^n}\right)+
   \int_{Q_{k+1}\setminus Q_k} f_\lambda(\xi+\varphi_kw_{\lambda,k})\d x\right)\,.
\end{equation}
Using \eqref{eq:8}, we observe that
\[
  \int_{Q_{k+1}\setminus Q_k} f_\lambda(\xi+\varphi_kw_{\lambda,k})\d x
  \leq C\int_{Q_{k+1}\setminus Q_k} \left(f_\lambda(\xi)+\varphi_k f_\lambda(w_{\lambda,k})\right)\d x\,,
\]
Here the second error term on the right hand side, when summed over $k$, can be estimated as follows,
\[
  \begin{split}
    \limsup_{\lambda\to\infty}\sum_{k=1}^L\int_Q \varphi_kf_\lambda(w_{\lambda,k})\d x&\lesssim \limsup_{\lambda\to\infty}\int_Q f_\lambda(w_{\lambda,k})\d x\\
    &\lesssim\limsup_{\lambda\to\infty}\left(\|w_{\lambda,k}\|_{L^1}+\lambda^{-q}\|w_{\lambda,k}\|_{L^p}^p\right)\\
    &\lesssim C\,.
\end{split}
\]
 Summing \eqref{eq:7} from $k=1$ to $L$ and dividing by $L$ yields
 \[
   \begin{split}
   |Q_1| f(\xi)
   &\leq \liminf_{\lambda\to\infty} \left(\int_{Q} f_\lambda\left(\frac{\d\mu_\lambda}{\d\L^n}\right)\d x+CL^{-1}\sum_{k=1}^L  \int_{Q_{k+1}\setminus Q_k} f_\lambda(\xi+\varphi_k\eta_\e*(\mu_\lambda-\mu))\d x\right)\\
   &\leq \liminf_{\lambda\to\infty} \left(\int_{Q} f_\lambda\left(\frac{\d\mu_\lambda}{\d\L^n}\right)\d x+CL^{-1}  \right)\,.
 \end{split}
   \]
Taking first the limit $L\to\infty$ and  then 
 $\delta\to 0$ (i.e., $|Q_1|\to |Q|=1$) we obtain the claim of the proposition. 
\end{proof}

\red{
\begin{remark}
  \begin{itemize}
  \item[(i)] The blowup technique (in the context of lower semicontinuity of integral functionals) that we use here has been developed  by Fonseca and M\"uller, first for  Sobolev functions \cite{MR1177778}, then for $BV$ functions \cite{MR1218685}, then for $\mathcal A$-free $L^p$ functions \cite{MR1718306}. The paper
    \cite{fonseca2004quasiconvexity} discusses lower semicontinuity for $\mathcal A$-free functions  in the weak-* convergence of measures; it is proved there in particular  that
    \[
      \int_\Omega f\left(x,\frac{\d\mu}{\d\L^n}(x)\right)\d x\leq \liminf_{k\to\infty} \int_\Omega f(x,v_k(x))\d x
    \]
    for any sequence $v_k\in L^1$ that converges weakly-* in the sense of measures to some $\R^d$ valued Radon measure $\mu$, 
   where  $f:\Omega\times \R^d\to\R$ is $\A$-quasiconvex in the second argument with linear growth at infinity, see Theorem 1.4 in \cite{fonseca2004quasiconvexity}. (Some additional regularity is required of $f$, which we omit here for the sake of brevity.) For $f=f_\lambda$, our Proposition \ref{prop:lowerreg} is a direct consequence of this theorem.

 
\item[(ii)]  Note that we use the $\mathcal A$-quasiconvexity of $f_\lambda$, and not any
  convexity properties of the limit $f$ to show our claim.
\end{itemize}

\end{remark}
}

\begin{lemma}
  \label{lem:lowerreg}
  Let   $\mu_\lambda$ be a sequence in  $\mathcal M(\overline Q;\Rsym)$ with $|\mu_\lambda|\ll\L^n$, $\d\mu_\lambda/\d\L^n\in L^2(Q;\Rsym)$, and $\xi\in \Rsym$, $\mu=\xi\L^n$ such that
            \[
              \begin{split}
                \mu_\lambda-\mu&\wsto 0\quad \text{ in }\mathcal M(\overline Q;\Rsym)\\
                \div(\mu_\lambda-\mu)&\wsto 0\quad \text{ in }\mathcal
                M(\overline Q;\Rsym)\\
                \lambda^{-1/4}\div(\mu_\lambda-\mu)&\to 0 \quad \text{ in }
                W^{-1,2}(Q;\Rsym)\\
                \lambda^{-1/4}\frac{\d\mu_\lambda}{\d\L^n}& \quad \text{ is bounded in }L^2(Q;\Rsym)\,.
              \end{split}
              \]
              Then
              \[
                h(\mu)\leq \liminf_{\lambda\to\infty}\int_Q h_\lambda(\mu_\lambda)\d x\,.                \]
\end{lemma}
\begin{proof}
  We apply Proposition \ref{prop:lowerreg} with $f_\lambda=h_\lambda$,
  $q=\frac12$ and $p=2$. The estimates \eqref{eq:2}, \eqref{eq:8} for
  $h_\lambda$ are easily verified by an explicit calculation.
\end{proof}

\begin{proof}[Proof of the lower bound in Theorem \ref{thm:main}] After choosing
  a suitable subsequence, we may assume that the $\liminf$ is a limit. We
  recall that $h_\lambda(\mu_\lambda)=Q\tilde h_\lambda(\mu_\lambda)\leq \tilde
  h_\lambda(\mu_\lambda)$. Hence $h_\lambda(\mu_\lambda)\L^n$ is a bounded
  sequence in $\mathcal M(\overline\Omega;\Rsym)$. 
After passing to a further subsequence, we have that $h_\lambda(\mu_\lambda)\L^n\wsto
\pi$ for some $\pi\in\mathcal M(\overline\Omega)$, with
\[
  \pi(\overline\Omega)=\lim_{\lambda}\int_\Omega h_\lambda(\mu_\lambda)\d x\leq
  \lim_{\lambda}\int_\Omega \tilde h_\lambda(\mu_\lambda)\d x\,.
\]
Since $\mu^s=\mu-\frac{\d\mu}{\d\L^n}\L^n$ is singular with respect
to $\L^n$, we have 
that 
\[
  \pi\geq \frac{\d\pi}{\d\L^n}\L^n +\frac{\d\pi}{\d|\mu^s|}{|\mu^s|}\,.
\]
By a well known representation of $W^{-1,p}$ (see
e.g. \cite{MR1014685} Theorem 4.3.3), we may write, in the sense of distributions,
\[
  \begin{split}
    g_\lambda=\alpha_\lambda+\div\beta_\lambda\\
  \end{split}
  \]
with $\alpha_\lambda\in L^{2}(\R^n;\R^n)$ and 
$\beta_\lambda\in L^{2}(\R^n;\Rsym)$.
  We recall that $\lambda^{-1/4}g_\lambda\to 0$ in $W^{-1,2}(\R^n)$ by
  assumption. This implies that $\alpha_\lambda$, $\beta_\lambda$ may be chosen
  such that
    \begin{equation}\label{eq:23}
    \lambda^{-1/4}\left(\|\alpha_\lambda\|_{L^2(\R^n)}+\|\beta_\lambda\|_{L^2(\R^n)}\right)\to
    0\,.
  \end{equation}
We have that for $\L^n$ almost every $x_0$, $T^{x_0,r}_\# \mu\wsto
\d\mu/\d\L^n(x_0)\L^n$, see Section \ref{sec:tangent-measures}. For fixed $r$,
we have that $T^{x_0,r}_\# \mu_\lambda\wsto T^{x_0,r}_\# \mu$. Hence we may
choose a sequence $r_\lambda\downarrow 0$ such that 
  \begin{equation}\label{eq:20}
\bar\mu_\lambda:=T^{x_0,r_\lambda}_\# \mu_\lambda\wsto \frac{\d\mu}{\d\L^n}(x_0)\L^n
\end{equation}
In the same way we may assume 
  \begin{equation}\label{eq:21}
\bar g_\lambda:=T^{x_0,r_\lambda}_\# g_\lambda\wsto \frac{\d g}{\d\L^n}(x_0)\L^n\,.
\end{equation}
By the Radon-Nikod\'ym Theorem,
we also have (again, for $\L^n$ almost every $x_0$)
\[
\frac{\d\pi}{\d\L^n}(x_0)=
                \lim_{\lambda\to\infty}\frac{1}{\L^n(Q(x_0,r_\lambda))}\int_{Q(x_0,r_\lambda)}h_\lambda\left(\frac{\d\mu_\lambda}{\d\L^n}\right)\d
                x\,.
\]

Now we verify for an $x_0$ that satisfies  the above relations  that the
conditions of Lemma \ref{lem:lowerreg} are fulfilled for the sequence
$\bar\mu_\lambda:=T^{(x_0,r_\lambda)}_\#\mu_\lambda$. The first condition of
that lemma is just  \eqref{eq:20}.
Furthermore,  note that
\[
  -\div \bar\mu_\lambda=r_\lambda \bar g_\lambda\,,
\]
and hence we obtain by \eqref{eq:21} that
\[
  \div\bar\mu_\lambda\wsto 0\,,
  \]
which is the second condition of Lemma \ref{lem:lowerreg}.

Setting
\[
  \begin{split}
    \bar \alpha_\lambda(x)&:=\alpha_\lambda(x_0+r_\lambda x)\\
    \bar \beta_\lambda(x)&:=\beta_\lambda(x_0+r_\lambda x)
  \end{split}
\]
we have by \eqref{eq:23} (assuming that
$r_\lambda^{-1}\lambda^{-1/2}\left(\|\alpha_\lambda\|_{L^2}^2+\|\beta_\lambda\|_{L^2}^2\right)\to
0$, which may be achieved by possibly modifying the sequence $r_\lambda$)
  \[
    \begin{split}
      \|\lambda^{-1/4}\div\bar\mu_\lambda\|_{W^{-1,2}(Q)}&\lesssim
      \lambda^{-1/4}\left(\|\bar\alpha_\lambda\|_{L^2(Q)}+r_\lambda \|\bar \beta_\lambda\|_{L^2(Q)}\right)\\
            &\to
      0
    \end{split}
  \]
      This is just  the third condition of Lemma \ref{lem:lowerreg}.

Finally we observe  that
  \[
    \begin{split}
    \int_Q \lambda^{-1/2}\left|\frac{\d\bar\mu_\lambda}{\d\L^n}\right|^2\d x&\lesssim
    \int_{\{\rho^{(n)}(\mu_\lambda)\leq \sqrt{\lambda}\}} \left|\frac{\d\bar\mu_\lambda}{\d\L^n}\right|\d x +
    \int_{\{\rho^{(n)}(\mu_\lambda)\geq \sqrt{\lambda}\}} \left(\lambda^{-1/2}\left|\frac{\d\bar\mu_\lambda}{\d\L^n}\right|^2+\lambda^{1/2}\right)\d x\\
    &\lesssim \int_Q h_\lambda\left(\frac{\d\bar\mu_\lambda}{\d\L^n}\right)\d x\\
    &<C
  \end{split}
\]
which proves boundedness of $\lambda^{-1/4}{\d\bar\mu_\lambda}/{\d\L^n}$ in $L^2$, the last
condition in Lemma \ref{lem:lowerreg}.

The application of Lemma \ref{lem:lowerreg} yields
    \begin{equation}
\label{eq:11}   \frac{\d\pi}{\d\L^n}(x_0)\geq h\left(\frac{\d\mu}{\d\L^n}(x_0)\right)\,.
  \end{equation}

  \medskip

  For $|\mu^s|$ almost every $x_0\in\overline\Omega$, we have that
  \[
    \begin{split}
      \frac{\d\pi}{\d |\mu^s|}(x_0)&=\lim_{r\to
        0}\frac{1}{|\mu|(Q(x_0,r))}\lim_{\lambda\to\infty}\int_{Q(x_0,r)}h_\lambda\left(\frac{\d\mu_\lambda}{\d\L^n}\right)\d x\\
      &\geq h\left(\frac{\d \mu}{\d|\mu|}(x_0)\right)\,,
      \end{split}
    \]
    where the last inequality is obtained by Proposition
    \ref{prop:lowersing}. 

    Hence we have shown
    \[
      \begin{split}
        \pi(\overline\Omega)&\geq \int_\Omega h\left(\frac{\d\mu}{\d\L^n}\right)\d x
        + \int_{\overline\Omega} h\left(\frac{\d\mu}{\d|\mu|}\right)\d|\mu^s|\\
        &=\int_{\overline\Omega} h\d\mu\\
        &=\mathcal G_{\infty,g}(\mu)\,.
      \end{split}
    \]
    This completes the proof of the lower bound.
   \end{proof}

  \section{Compactness, upper bound}
  \begin{proof}[Proof of compactness in Theorem \ref{thm:main}]
We have that $|\mu_\lambda|\leq h_\lambda(\mu_\lambda)\leq \tilde h_\lambda(\mu_\lambda)$, and hence the
statement follows from the standard compactness result for sequences in $\M(\overline\Omega;\Rsym)$ in
the weak * topology.
  \end{proof}

\begin{proof}[Proof of the upper bound in Theorem \ref{thm:main}] We may assume
  that $\mathcal G_{\infty,g}(\mu)<\infty$, otherwise there is nothing to show.
We consider $g_\lambda,g$ as
measures in $\mathcal M(\R^n;\R^n)$ with support in $\overline\Omega$. 

  We observe that $\mathcal M(\R^n;\R^n)\subset W^{-1,
    p}(\R^n;\R^n)$ for $p\in
(1,\frac{n}{n-1})$ with compact embedding. 
Now we
 apply standard results for strongly elliptic equations
with constant coefficients:
 Let $\zeta_\lambda\in W ^{1,p}_{\mathrm{loc}}(\R^n;\R^n)$ be the  solution of
  \[
    \left\{\begin{split}
        -\div e(\zeta_\lambda)&=g_\lambda\quad\text{ in }\R^n\\
        \nabla \zeta_\lambda&\in L^p(\R^n;\R^{n\times n})
        \end{split}\right.
  \]
  where $e(\zeta_\lambda)=\frac12(\nabla \zeta_\lambda+\nabla\zeta_\lambda^T)$.
 The application of elliptic  regularity theory  yields
  \[
    \|\nabla \zeta_\lambda\|_{L^p(\R^n)}\lesssim \|g_\lambda\|_{W^{-1,p}(\R^n)}\,.
    \]
In the same way, we obtain a solution $\zeta$ of  
\[
    \left\{\begin{split}
        -\div e(\zeta)&=g\quad\text{ in }\R^n\\
        \nabla \zeta&\in L^p(\R^n;\R^{n\times n})
    \end{split}\right.
  \]
with
\[
  \|\nabla \zeta\|_{L^p(\R^n)}\lesssim \|g\|_{W^{-1, p}(\R^n)}\,.
  \]

\medskip
  

\medskip

  By the assumption $g_\lambda\wsto g$, the compact embedding $\mathcal M\subset
  W^{-1, p}$ and elliptic regularity,
  we have that
  \[
    e(\zeta_\lambda)\to e(\zeta)\quad\text{ in } L^{ p}(\R^n;\Rsym)\,.
  \]
 By $\lambda^{-1/4}\|g_\lambda\|_{W^{-1,2}(\R^n;\R^n)}\to 0$ and
 elliptic regularity, we have that
  \[
    \begin{split}
      \lambda^{-1/4}\|e(\zeta_\lambda)\|_{L^2(\R^n;\Rsym)}&\to 0\,.
    \end{split}
    \]
Set $\bar\mu=\mu-e(\zeta)$. Let $\eta\in C_c^\infty(\R^n)$ such that $\int \eta=1$
and $\eta_\e:=\e^{-n}\eta(\cdot/\e)$. Choose a monotone decreasing sequence
$\e(\lambda)$ with $\e(\lambda)\downarrow 0$ as $\lambda\to\infty$ and
$\frac{|\mu|}{\e(\lambda)^{n}}\sup|\eta| \leq \frac14\sqrt{\lambda}$. We set 
\[
  \bar\mu_\lambda:=\eta_\e*\bar\mu
\]
and
\[
  \mu_\lambda=\bar\mu_\lambda+e(\zeta_\lambda)\,.
\]
Note that these definitions imply in particular that
$|\bar \mu_\lambda|\leq \frac14\sqrt{\lambda}$. Furthermore  we have that
$\mu_\lambda\wsto \mu$.


Let
\[
  \begin{split}
     A_\lambda&:=\{x\in \Omega: \rho^{(n)}(\mu_\lambda)\geq \sqrt\lambda\}\\
    \tilde A_\lambda&:=\{x\in \Omega: |e(\zeta_\lambda)|\geq \frac14\sqrt\lambda\}\,.
  \end{split}
\]
By $\rho^{(n)}(\xi)\leq 2|\xi|$ for all $\xi \in \Rsym$, we have that
$A_\lambda\subset\tilde A_\lambda$.
Now we may estimate as follows (not distinguishing between measures and their densities),
\[
  \begin{split}
    \limsup_{\lambda\to\infty} \int_\Omega h_\lambda(\mu_\lambda)\d x
    &=\limsup_{\lambda\to\infty} \left(\int_{\Omega\setminus A_\lambda} h_\lambda(\mu_\lambda)\d x+\int_{A_\lambda} h_\lambda(\mu_\lambda)\d x\right)\\
    &\leq \limsup_{\lambda\to\infty} \int_{\Omega\setminus A_\lambda} h(
     \mu_\lambda)\d x+ \int_{A_\lambda}
     \left(\frac{\mu_\lambda^2}{\sqrt{\lambda}}+\sqrt{\lambda}\right)\d x\,.
   \end{split}
 \]
 Now we have that
$\lambda^{-1/2}\|e(\zeta_\lambda)\|_{L^2}^2\to 0$ and hence 
$\lambda^{1/2}\L^n(\tilde A_\lambda )\to 0$ as $\lambda\to \infty$. This implies
\[
  \begin{split}
    \int_{A_\lambda}
    \left(\frac{\mu_\lambda^2}{\sqrt{\lambda}}+\sqrt{\lambda}\right)\d x\, &\leq
    \int_{\tilde A_\lambda}
    \left(2\frac{\bar\mu_\lambda^2+|e(\zeta_\lambda)|^2}{\sqrt{\lambda}}+\sqrt{\lambda}\right)\d
    x\\
    &\leq \L^n(\tilde A_\lambda)\left(\frac{1}{8}\sqrt{\lambda}+\sqrt{\lambda}
    \right)+\lambda^{-1/2}\|e(\zeta_\lambda)\|_{L^2}^2\\
    &\to 0\quad\text{ as }
    \lambda\to \infty\,.
  \end{split}
  \]
Hence we get
\[
  \begin{split}\limsup_{\lambda\to\infty} \int_\Omega h_\lambda(\mu_\lambda)\d x
      &\leq \limsup_{\lambda\to\infty} \int_\Omega h(\mu_\lambda)\d x\\
    &=\int_\Omega \d h( \mu)\\
    &=\G_{\infty,g}(\mu)\,.
  \end{split}
  \]
\end{proof}

\appendix

\section{Derivation of the variational  form of the compliance minimization problem}
\label{sec:derivation-variational}
Here we repeat basically our presentation from Section 2.4 of \cite{olbermann2017michell}. We include this part in order to keep the present article self-contained.

\medskip

Let $g\in W^{-1,2}(\overline \Omega;
\R^n)$. 
The aim of this appendix is to  give a  derivation of the compliance minimization problem in its
variational form, 
  \begin{equation}
  \inf\G_{\lambda;g}(\sigma)\,,\label{eq:15}
  \end{equation}
where the infimum is taken over the set
\[
  S_g(\Omega)=\{\sigma \in L^2(\Omega;\Rsym):-\div\sigma=g\}
  \]
Here the equation $-\div\sigma=g$ is to be understood as an equation in the
distributional sense in a
neighborhood of $\overline\Omega$, with $\sigma$ extended by 0 on the complement
of $\Omega$. In this way we incorporate boundary conditions in the equation, see
our discussion in Section \ref{sec:notation}.
  We want to derive this variational problem starting from the standard
formulation of a linear elasticity problem. More details can be found in  \cite{MR1859696}.

\medskip

Consider  $\Omega\subset\R^n$ as an elastic body, characterized by its
elasticity tensor $A_0\in  \mathrm{Lin}(\Rsym;\Rsym)$, where for simplicity we
assume here that
$A_0=\mathrm{Id}_{\Rsym}$ is the identity.
  We remove a subset $H\subset \Omega$ from the elastic body and the new boundaries from that
process shall be traction-free. The resulting linear elasticity problem is to
find $u:\Omega\setminus H\to\R^n$ such that
\[
\begin{split}
  \sigma&=A_0 e(u)\\
  -\div \sigma&=g\quad\text{in }\overline\Omega \setminus \overline H\\
 \sigma\cdot n&=0\quad \text{on
  }\partial H\,,
\end{split}
\]
where $e(u)=\frac12(\nabla u+\nabla u)^T$. 
The compliance (work done by the load) is given by 
\[
c(H)=\int_{\partial\Omega}g\cdot u\d\H^1=\int_{\Omega\setminus H}(A_0
e(u)):e(u)\d x\,,
\]
where $u:\Omega\setminus H\to \R^2$ is the unique solution to the linear elasticity system above.
We want to minimize the compliance under a constraint on the ``weight'' $\L^2(\Omega\setminus H)$.
We do so by the introduction of a Lagrange multiplier $\lambda$, and are
interested in the minimization problem
\[
\min_H \left(c(H)+\lambda \L^2(\Omega\setminus H)\right)\,.
\]
The ``equivalence'' between the mass constrained problem and the problem
including a Lagrange multiplier only holds on a heuristic level, see
\cite{MR820342} for a discussion of this point.
Accepting this step, taking the limit of vanishing
weight corresponds  to the limit
$\lambda\to\infty$. 
We now rewrite the problem by  considering space-dependent elasticity  tensors 
of the form $A(x)=\chi(x)A_0$, where $\chi\in L^\infty(\Omega;\{0,1\})$.
The equations from above turn into the system

\begin{equation}
\begin{split}
  \sigma=&A(x) e(u)\\
  -\div \sigma=&g\quad\text{in }\overline\Omega
\end{split}\label{eq:55}
\end{equation}
The compliance turns into a functional on the set of permissible elasticity tensors,
and is given by 
\[
c(A)=\int_{\Omega}(A(x)
e(u)):e(u)\d x\,,
\]
where $u$ is the solution of \eqref{eq:55}.
By the principle of minimum complementary energy,  the compliance
can be written as
\[
c(A)=\int_\Omega G(A(x),\sigma(x))\d x\,,
\]
where
\[
G(\bar A,\xi)=\begin{cases} +\infty & \text{ if }\xi\neq 0 \text{ and }\bar A=0\\
0 & \text{ if }\xi= 0 \text{ and }\bar A=0\\
(\bar A^{-1}\xi):\xi&\text{ else, } \end{cases}
\]
and $\sigma\in L^\infty(\Omega;\Rsym)$ is a solution of the PDE
\[
\begin{split}
  -\div\sigma&=g\quad\text{ in }\overline\Omega\,,
\end{split}
\]
i.e., $\sigma\in S_g(\Omega)$.
We see that the compliance minimization problem can be understood as the
variational problem of finding the infimum 
\[
\inf\left\{ \int_\Omega \left(G(\chi(x) A_0,\sigma(x))+\lambda \chi(x)\right)\d x:
  \chi\in L^\infty(\Omega;\{0,1\}),\,\sigma\in S_g(\Omega)\right\}\,.
\]
Of course, the compliance of a pair $(\chi,\sigma)$ is infinite if  there exists
a set of 
positive measure $U$ such that $\chi=0$ and $\sigma\neq 0$ on $U$. Hence the
above variational problem is equivalent with 
\begin{equation}
\inf\left\{ \int_\Omega F^{A_0}_\lambda(\sigma)\d x:\sigma\in S_g(\Omega)
  \right\}\,,\label{eq:57}
\end{equation}
where
\begin{equation*}
F^{A_0}_\lambda(\xi)=\begin{cases}0&\text{ if }\xi=0\\
(A_0^{-1}\xi):\xi+\lambda&\text{ else,}\end{cases}
\end{equation*}
Up to a factor $\lambda^{-1/2}$, this is just the integrand \eqref{eq:13}, and
hence \eqref{eq:57} is just the variational problem \eqref{eq:15}, with $A_0=\mathrm{Id}_{\Rsym}$.
As is well known, this problem  does not possess a
solution in general and requires relaxation.

\red{
\section{A very brief presentation of Michell trusses}

\label{sec:very-brief-pres}

In this section, we want to sketch very briefly how the limit integral functional
$\G_{\infty,g}$ is linked to Michell truss theory for the case $n=2$. What we say here is mainly taken
from \cite{bouchitte2008michell}.

\medskip

A \emph{truss} is a finite union of \emph{bars} (line segments that
      can resist compression or tension parallel to them) between points
      $x_i\in\R^2$,      $i=1,\dots,M$.
      We write $(x_1,\dots,x_m)=x\in \R^{2\times M}$, and let $w\in \R^{M\times M}_{\mathrm{sym}}$.
To every bar $[x_i,x_j]=\{tx_i+(1-t)x_j:t\in[0,1]\}$, we associate 
$w_{ij}$, where $|w_{ij}|$ is the \emph{strength} of the bar, and the sign of
$w_{ij}$ is chosen according to whether the bar has to withstand compression or tension. 

      The force provided the bar $[x_i,x_j]$ is given by
    \[
    f_{ij}(x,w)=\left(\delta_{x_i}-\delta_{x_j}\right)w_{ij}\frac{x_i-x_j}{|x_i-x_j|} \,.   \]

  The set of all bars shall counterbalance a given force

\[
g= \sum_{i=1}^N g_i\delta_{y_i}
\]
where $y_i, g_i\in \R^2$, $i=1,\dots,N$, are given. The \emph{truss} $\left(\{x_i\}_{i=1,\dots,M},\{w_{ij}\}_{i,j=1,\dots,M}\right)$
  withstands $g$ if
\[
g+ \sum_{i,j=1,\dots,M} f_{ij}(x,w)=0\,.
\]
The \emph{weight} of the truss $(x,w)$ is given by 
\[
{\mathcal W}(x,w)=\sum_{i,j=1}^M  |w_{ij}||x_i-x_j|\,.
\]
The task is now the minimization of the weight, given the external forces, as a
function of $x,w$. To express how this variational problem relates to
$\G_{\infty,g}$, we note that the force supplied by the bars can be written as the divergence of a
  stress, $f_{ij}(x,w)=\div \sigma_{ij}(x,w)$ with 
\[
\begin{split}
  \sigma_{ij}(x,w)&= w_{ij}\frac{x_i-x_j}{|x_i-x_j|}\otimes
  \frac{x_i-x_j}{|x_i-x_j|} \H^1|_{[x_i,x_j]}\\
  \sigma(x,w)&=\sum_{i,j=1}^M \sigma_{ij}(x,w)\in \mathcal M(\R^2;\R^{2\times 2}_{\mathrm{sym}})\,.
\end{split}
\]
With this notation, the balance of forces becomes the equation

\[
  -\div\sigma=g\,,
\]
and the weight of the truss is given by $\mathcal
W(x,w)=|\sigma(x,w)|$ (the total variation of the measure $\sigma$).

\medskip

Summarizing, we are dealing with the variational problem
\[
\inf\left\{|\sigma|:\sigma=\sigma(x,w)\text{ for some truss
}(x,w),\,\,-\div\sigma=g\right\}\,.
\]

To guarantee the existence of a minimizer, this variational problem
requires relaxation, as has already been remarked by Michell in 1904
\cite{lviii185limits}. We will not discuss the derivation of the relaxation here
and refer the interested reader to \cite{bouchitte2008michell}. We only state
the result: Namely, that it becomes the variational problem defined by the
$\Gamma$-limit in the main text:

\[
\inf\left\{\rho^{(2)}(\sigma)(\R^2):\sigma\in \mathcal{M}(\R^2;\Rsym),\,\,-\div\sigma=g\right\}\,.
\] 
 Requiring additionally $\mathrm{supp}\,\sigma\subset\overline\Omega$ leads
  to
\[
\inf\left\{\rho^{(2)}(\sigma)(\overline{\Omega}):\sigma\in \mathcal{M}(\overline{\Omega};\Rsym),\,\,-\div\sigma=g\right\}\,.
\] 



}

    \bibliographystyle{alpha}
\bibliography{michell3d}

\end{document}